\documentclass[reqno]{amsart}
\usepackage{amsmath}
\usepackage{amsfonts}
\usepackage{amstext}
\usepackage{amsbsy}
\usepackage{amsopn}
\usepackage{amsxtra}
\usepackage{upref}
\usepackage{amsthm}
\usepackage{amsmath}
\usepackage{amssymb}
\usepackage{euscript}
\usepackage{graphicx}
\usepackage[cspex,bbgreekl]{mathbbol}
\usepackage{enumerate}
\usepackage[textures,bookmarks=false]{hyperref}
\usepackage{mathrsfs}

\newtheorem{prop}{Proposition}[section]
\newtheorem{rem}{Remark}[section]
\newtheorem{lema}{Lemma}[section]
\newtheorem{defi}{Definition}[section]
\newtheorem{teo}{Theorem}[section]

\newtheorem{eje}{Example}[section]

\newtheorem*{claim*}{Claim}

\def\c{{\bf C}}
\def\eps{\varepsilon}
\def\phi{\varphi}
\def\R{{\mathbb R}}

\def\N{{\mathbb N}}

\def\P{{\mathcal P}}

\def\F{{\mathcal F}}

\def\M{{\mathcal M}}

\def\es{{\emptyset}}
\def\sm{\setminus}

\def\crit{{\mathcal Cr}}

\def\bd{\partial }
\def\le{\leqslant}
\def\ge{\geqslant}
\def\st{such that }

\def\F{\mathcal{F}}
\def\M{\mathcal{M}}

\def\trans{\mathcal{T}}

\title[Transience in Dynamical Systems]{Transience in Dynamical Systems}
\date{\today}

\begin{thanks}
{ G.I. was partially  supported by Proyecto Fondecyt 11070050. MT was partially supported by FCT grant SFRH/BPD/26521/2006, by FCT through CMUP and by NSF grants DMS 0606343 and DMS 0908093}
\end{thanks}


\author{Godofredo Iommi} \address{Facultad de Matem\'aticas,
Pontificia Universidad Cat\'olica de Chile (PUC), Avenida Vicu\~na Mackenna 4860, Santiago, Chile}
\email{giommi@mat.puc.cl}
\urladdr{http://www.mat.puc.cl/\textasciitilde giommi/}
\author{Mike Todd} \address{
Mathematical Institute,
University of St Andrews,
North Haugh,
St Andrews,
KY16 9SS,
Scotland }
\email{mjt20@st-andrews.ac.uk }
\urladdr{http://www.mcs.st-and.ac.uk/~miket/}

\begin{document}

\begin{abstract}
We extend the theory of transience to general dynamical systems with no Markov structure assumed.   This is linked to the theory of phase transitions.  We also provide new examples to illustrate different kinds of transient behaviour.
\end{abstract}

\maketitle

\section{Introduction}
\label{sec:intro}

A classical question in the theory of random walks, Markov chains and ergodic theory is whether a system is recurrent or transient.  The notion of recurrence and its consequences  are well understood in the former two cases \cite{fe,v1,v2}.  Whereas  in the realm of ergodic theory it is not even clear what good definitions of recurrence or transience are.  
The definition involves a triple $(X,f, \phi)$, where $f:X \to X$ is a dynamical system and $\phi:X \to \R$ is a function (or potential). 
In the context of countable Markov shifts,  with a fairly weak assumption on the smoothness of the potential,   Sarig \cite{Sarnull} gave a definition of recurrence which comes naturally from the theory of Markov chains.  This definition even applies in situations such as those studied in \cite{Hnonuni, HofKel_pw,ManPom}, which can be thought of as non-uniformly hyperbolic, although the dynamical system still has a well-defined Markov structure.
The aim of this paper is to provide a definition of transience which applies in a much more general context. In particular, the definition we propose (see Definition \ref{def:trans_new}) requires no Markov structure for $(X,f)$, therefore we can not make use of the corresponding notions for random walks or Markov chains.  We also only need very weak assumptions on the smoothness of potentials.

In Section~\ref{sec:no conf}  we show that our definition can be checked even for certain non-uniformly hyperbolic dynamical systems with no canonical Markov structure.  We particularly focus on multimodal interval maps with the so-called `geometric potentials',  since this class of systems combines both a highly non-Markov structure and quite singular potentials.  This analysis passes to many other classes of interval maps.  



Our definition of recurrence/transience makes use of thermodynamic formalism.  Given a system $(X,f,\phi)$, one can study the recurrence/transience within the family of potentials $\{t\phi:t\in \R\}$.  In the cases we know, the transition at a point $t=t_0$ of $t\phi$ from recurrence to transience coincides with the pressure function $t\mapsto P(t \phi)$ not being real analytic at $t_0$.  This lack of real analyticity is referred to as a \emph{phase transition} at $t_0$. In Section \ref{sec:ex} we examine the transition from recurrence to transience in certain systems, giving some explicit and fairly elementary examples to illustrate the range of possible forms these transitions may take. 

\section{Definition of transience}
\label{sec:trans}
Our definition of recurrence/transience uses the ideas of pressure as well as conformality and conservativity of measures for a given dynamical system.  We first introduce pressure. 
   Throughout we will be dealing with metric spaces $X$, and Borel functions $f:X\to X$, our \emph{dynamical systems}.  We define $\M_f$ to be the set of $f$-invariant Borel probability measures.   Given a Borel function $\phi:X\to \R$ (the \emph{potential}), we consider the \emph{pressure} to be
$$P(\phi):=\sup \left\{ h(\mu) + \int \phi \ d \mu : \mu \in \M_f \textrm{ and } -  \int \phi \ d \mu< \infty \right\},$$
where $h(\mu)$ denotes the entropy of the measure $\mu$. A measure $\nu \in \M_f$ attaining the above supremum is called an \emph{equilibrium measure/state} for $(X,f,\phi)$.  For many systems, the existence of an equilibrium state is a sufficient condition for recurrence.  However, as explained below, there are many examples of systems $(X,f, \phi)$ with an equilibrium state, but which should not be considered recurrent.

The main tool used to study pressure, equilibrium states, and, as we will see, recurrence, is the \emph{Transfer (or Ruelle) operator}, which is defined by \begin{equation}
(L_{\phi}g)(x)=\sum_{Ty=x} g(y) \exp(\phi(y))
\end{equation}
for $x\in X$, $\phi$ and $g$ in some suitable Banach spaces (constructing suitable Banach spaces where this operator acts and behaves well is an important line of research in the field).  As we will see in  the classical result Theorem~\ref{thm:Ru}, in good cases, there exists $\lambda>0$, a Borel probability measure $m$ on $X$ and a continuous function $h:X\to [0,\infty)$ such that
\begin{equation}L_{\phi}  h = \lambda h \textrm{ and }  L_{\phi}^* m = \lambda m,\label{eq:h and m}
\end{equation}
where $\lambda=e^{P(\phi)}$ and $L_{\phi}^*$ is the dual operator of $L_{\phi}$ (i.e, for a continuous function  $\psi:X\to \R$, then  $\int\psi~dm=\int L_\phi \psi~dm$).  If, moreover, $ \int h \ dm< \infty$ then
the measure $\mu:=\frac{hm}{\int h~dm}$ is in $\M_f$.  Such measures $m$ will occupy an important place in this work and are the subject of the following definition.

%

\begin{defi}
Given $f:X \to X$, a dynamical system and $\phi: X \to \R$ a potential,  we say that a Borel measure $m$ on $X$ is \emph{$\phi$-conformal} if
$$L_{\phi}^*m=m$$
and there exists an at most countable collection of open sets $\{U_i\}_i\subset X$ such that $m\left(X\sm \cup_i U_i\right)=0$ and $m(U_i)<\infty$ for all $i$.
\end{defi}


\begin{rem}
Notice that by this definition, $m$ in \eqref{eq:h and m} is $(\phi-\log\lambda)$-conformal.  Also we point out that some authors would consider such $m$ to be $\phi$-conformal; unless $\lambda=1$, we do not.

Also note that if $f^n:U\to f^n(U)$ is 1-to-1 then $m(f^n(U))=\int_Ue^{-S_n\phi}~dm$ where
 $$S_n\phi(x):=\phi(x)+\cdots +\phi\circ f^{n-1}(x).$$

A dynamical system is \emph{topologically exact}, if for any open set $U\subset X$ there exists $n\in\N$ such that $f^n(U)=X$.  If our system is topologically exact and $\phi$ is a bounded potential, then by the above $m(X)<\infty$.  Whenever the measure of $X$ is finite then we normalise.
\end{rem}

We will define $(X,f, \phi)$ to be recurrent if there is a $(\phi-P(\phi))$-conformal measure which has the properties outlined below.

A measure $\mu$ on $X$ is called  $f$-\emph{non-singular} if $\mu(A)=0$ if and only if $\mu(f^{-1}(A))=0.$   A set $ W \subset X$ is called $\emph{wandering}$ if  the sets $\{f^{-n}(W)\}_{n=0}^{\infty}$ are disjoint.

\begin{defi} \label{def:con} Let $f: X \to X$ be a dynamical system. An $f$-non-singular measure $\mu$ is called \emph{conservative} if every wandering set $W$  is such that
$\mu(W)=0$.
\end{defi}
A conservative measure satisfies the Poincar\'e Recurrence Theorem (see \cite[p.17]{Aaro_book}, or \cite[p.31]{sanotes}).

A further property we would like any set of recurrent points to satisfy is given in the following definition. 
\begin{defi}
Let $B_{\eps}(x_0)$ denote the open ball of radius $\eps$ centred at the point $x_0$.
We say that $x\in X$ \emph{goes to $\eps$-large scale} at time $n$ if there exists an open set $U\ni x$ such that $f^n: U \to B_\eps(f^n(x))$ is a bijection.  We say that $x$ \emph{goes to $\eps$-large scale infinitely often} if there exists $\eps>0$ such that $x$ goes to $\eps$-large scale for infinitely many times $n\in \N$.  Let $LS_\eps\subset X$ denote the set of points which go to $\eps$-large scale infinitely often.
\end{defi}

\begin{defi}
Given a Borel measure $\mu$ on $X$ we say that $\mu$ is \emph{weakly expanding} if there exists $\eps>0$ such that $\mu(LS_\eps)>0$.
\end{defi}

We use the term `weakly expanding' for our measures to distinguish from the expanding measures in \cite{Pinheiro}  (note that those measures go to large scale with \emph{positive frequency}).  As we note in Remark~\ref{rmk:shift WE}, for a dynamical system with a Markov structure, any conservative measure is automatically weakly expanding.
We are now ready to define recurrence.

\begin{defi} \label{def:trans_new}
 Let $f: X \to  X$ be a dynamical system and $\phi: X \to [-\infty, \infty]$ a Borel potential.  Then $(X,f, \phi)$ is called \emph{recurrent} if there is a finite weakly expanding conservative $(\phi-P(\phi))$-conformal measure $m$. Moreover, if  there exists a finite $f$-invariant measure $\mu\ll m$, then we say that $\phi$ is \emph{positive recurrent}; otherwise we say that $\phi$ is \emph{null recurrent}.

If the system $(X,f, \phi)$ is not recurrent it is called  \emph{transient}.
\end{defi}

Sarig defined recurrence in the setting of countable Markov shifts with fairly well-behaved potentials, as in \eqref{eq:Sar rec} below.  We will compare these two definitions of recurrence: in particular we give examples, outside the context used by Sarig, which the definition in \eqref{eq:Sar rec} couldn't handle and show that ours can handle them. An easily stated such example is the following.  

\begin{eje} \label{beta}
Let $\beta >1$ be a real number. The $\beta$-transformation is the interval map  $T_{\beta}:[0,1) \to [0,1)$ defined by $T_{\beta}(x) = \beta x \mod 1.$ This map is Markov for only countably many values of $\beta$. It was shown by Walters \cite{wbeta} that  for any $\beta>1$, if $\phi:[0,1) \to \R$ is a Lipschitz potential then there exists a finite weakly expanding conservative $(\phi-P(\phi))$-conformal measure $m$. Moreover, there exists a finite $T_{\beta}$-invariant measure $\mu\ll m$. That is, if $\phi$ is a Lipschitz potential then the triple $([0,1), T_{\beta}, \phi)$ is positive recurrent. 
\end{eje}

We will be particularly interested in systems $(X, f, \phi)$ where for some values of $t$, $(X, f, t\phi)$ is recurrent, and for others it is transient.  This phenomenon is associated to the smoothness of the pressure function $p_\phi(t):=P(t\phi)$.  If the function is not real analytic at some $t_0\in \R$ then we say there is a \emph{phase transition} at $t_0$.  If this function is not even $C^1$ at $t_0$ then we say that there is a \emph{first order phase transition} at $t_0$.  Phase transitions can also be associated with the non-uniqueness of equilibrium states, but this is not always the case, as shown in Section~\ref{ssec:ho-ke} (see the penultimate row of Figure~\ref{table:hof}). For our examples, we could alternatively characterise the point $t_0$ as: for any open interval $U\ni t_0$ there exist $t_1, t_2\in U$ such that system $(X, f, t_1\phi)$ is recurrent, and $(X, f, t_1\phi)$ is transient.  We will give further examples to show what can happen at phase transitions. 

We finish this section by asking some questions our definition of transience raises:

\begin{list}{$\bullet$}
{ \itemsep 1.0mm \topsep 0.0mm \leftmargin=7mm} 
\item Are there examples for which our definition of recurrence/transience conflicts with that in  \eqref{eq:Sar rec}?  In all of the examples we have here, if the definition given by  \eqref{eq:Sar rec} is well-defined, then so is ours and they coincide. 
\item Is our definition of transience really more widely applicable than that given by  \eqref{eq:Sar rec}?  Example~\ref{beta} already gives evidence that this is so.   We give further evidence in Section~\ref{sec:intervals}.  
\item We tend to see transience kick in after/before some phase transition, so if this were always the case then we could define transience in this way.  Is it true that given a system $(X,f,\phi)$ where $(X,f,t\phi)$ is recurrent for   \emph{all} $t<t_0$, and there is a phase transition at $t_0$ then the system is transient for all $t>t_0$?  We construct examples where this is not the case in Section~\ref{sec:ex}, see also \cite[Example 3]{Sarphase}. 
\item What does the existence of a dissipative $(t_0\phi-p_\phi(t_0))$-conformal measure tell us about a phase transition at $t_0$?
\end{list}

The structure of the paper is as follows

\begin{list}{$\bullet$}
{ \itemsep 1.0mm \topsep 0.0mm \leftmargin=7mm}
\item In Section~\ref{sec:symb}, the theory of recurrence and transience is presented in the best understood setting (primarily through the work of Cyr and Sarig), the Markov shift case, firstly in the finite alphabet case, and then in the countable.  We show how our definition of recurrence fits in with this case, give a very brief sketch of some examples, as well as discussing an alternative type of definition of transience due to Cyr and Sarig.  

\item In Section~\ref{sec:intervals}, examples of interval maps which exhibit both transient and recurrent behaviour are given.  The transient behaviour is due to either the lack of smoothness of the potentials or the lack of hyperbolicity of the underlying dynamical system.  Our definition of transience is shown to be well-suited to this setting.  Sarig's definition of recurrence/transience is not tractable in many of these cases.  We also discuss the transition from recurrence to transience here.  Our primary applications are to multimodal interval maps $f:I\to I$ with the geometric potential $-\log|f'|$, which are discussed in detail in Section~\ref{sec:no conf}. 

\item In Section~\ref{sec:ex}, we give a class of simple interval maps and fairly elementary potentials which exhibit a range of different behaviours at the transition from recurrence to transience.

\end{list}

\section{Symbolic spaces} \label{sec:symb}

In this section we discuss thermodynamic formalism in the context of Markov shifts.  We review some results concerning the existence and uniqueness of equilibrium measures.   We also discuss the regularity properties of the pressure function.  Many properties of Markov shifts defined in finite alphabets are different to those for a countable alphabet. The lack of compactness of the latter shifts is a major obstruction for the existence of equilibrium measures and can also result in transience.  Symbolic spaces are of particular importance, not only because of their intrinsic interest, but also because they provide models for uniformly and non-uniformly hyperbolic dynamical systems (see for example  \cite{bo1, Ratner,snunif}). 

Let $S \subset \N$ be the \emph{alphabet} and $\trans$  be a matrix $(t_{ij})_{S \times S}$ of zeros and ones (with no row and no column
made entirely of zeros). The corresponding \emph{symbolic space} is defined by
\[
\Sigma:=\left\{ x\in S^{\N_0} : t_{x_{i} x_{i+1}}=1 \ \text{for every $i
\in \N_0$}\right\},
\]
and the shift map is defined by $\sigma(x_0x_1 \cdots)=(x_1 x_2
\cdots)$. If the alphabet $S$ is finite we say that $(\Sigma,\sigma)$ is a \emph{finite
Markov shift}, if $S$ is (infinite) countable we say that $(\Sigma,\sigma)$ is a \emph{countable Markov shift}.   Given $n\ge 0$, the \emph{word} $x_0\cdots x_{n-1}\in S^n$ is called \emph{admissible} if $t_{x_{i} x_{i+1}}=1$ for every $0\le i\le n-2$.  We will always assume that $(\Sigma, \sigma)$ is  \emph{topologically mixing}, except in Section~\ref{ssec:non-trans} where the consequences of not having this hypothesis are discussed. This is equivalent to the following property: for each pair $a,b \in S$ there exists $N \in \N$ such that for every $n > N$ there is an admissible word $\underline{a}=(a_0 \dots a_{n-1})$ of length $n$ such that $a_0=a$ and $a_{n-1}=b$.
If the alphabet $S$ is finite this is also equivalent to the existence of an integer $N \in \N$ such that every entry of the matrix $\trans^N$ is positive.

We equip $\Sigma$ with the topology generated by the $n$-cylinder sets:
\[ C_{i_0 \cdots i_{n-1}}:= \{x \in
\Sigma : x_j=i_j \text{ for } 0 \le j \le n-1 \}.\] 
We let $C(\Sigma)$ be the set of continuous functions $\phi:I\to \R$.
Given a function
$\phi\colon \Sigma \to\R$, for each $n \ge 1$ we set
\[
V_{n}(\phi) := \sup \left\{|\phi(x)-\phi(y)| : x,y \in \Sigma,\
x_{i}=y_{i} \text{ for } 0 \le i \le n-1 \right\}.
\]
Note that $\phi :\Sigma \to \R$ is continuous if and only if $V_n (\phi) \to 0$. The regularity of the potentials that we consider is fundamental when it comes to proving existence of equilibrium measures as well as recurrence/transience.

\begin{defi}
We say that $\phi:\Sigma \to \R$  has \emph{summable variations} if
$\sum_{n=2}^{\infty} V_n(\phi)<\infty$. 
Clearly, if $\phi$ has summable variations then it is continuous.  We say that $\phi:\Sigma \to \R$ is \emph{weakly H\"older continuous} if $V_n(\phi)$ decays exponentially, that is there exists $C>0$ and $\theta \in (0,1)$ such that $V_n(\phi) < C \theta^n$  for all $n\ge 2$. If this is the case then clearly it has summable variations.
\end{defi}

Note that in this symbolic context, given any symbolic metric, the notions of H\"older and Lipschitz function are essentially the same (see \cite[p.16]{ParPol}).

We say that $\mu$ is a \emph{Gibbs measure} on $\Sigma$ if there exist $K,P\in \R$ such that for every $n\ge 1$, given an $n$-cylinder $C_{i_0 \cdots i_{n-1}}$,
$$\frac1K\le \frac{\mu(C_{i_0 \cdots i_{n-1}})}{e^{S_n\phi(x)-nP}}\le K$$
for any $x\in C_{i_0 \cdots i_{n-1}}$.  We will usually have $P=P(\phi)$.

\begin{rem}
\label{rmk:shift WE}
In the topologically mixing Markov shift $(\Sigma, \sigma)$ case, due to the Markov structure, and since for any conservative measure $m$ satisfies the Poincar\'e Recurrence Theorem, $m$-a.e.  point goes to large scale infinitely often, even in the countable Markov shift case. Hence in our definition of transience, we can drop the weakly expanding requirement.
\end{rem}

\subsection{Compact case} \label{co}
When the alphabet $S$ is finite, the space $\Sigma$ is compact. Moreover, the entropy map $\mu \mapsto h(\mu)$ is upper semi-continuous. Therefore, continuous potentials have equilibrium measures. In order to prove uniqueness  of such measures, regularity assumptions on the potential and a transitivity/mixing assumption on the system are required.
The following is the Ruelle-Perron-Frobenius Theorem (see \cite[p.9]{bo3} and \cite[Proposition 4.7]{ParPol}).
\begin{teo} \label{thm:Ru}
 Let $(\Sigma, \sigma)$ be a topologically mixing  finite Markov shift and let $\phi : \Sigma \to \R$ be a H\"older potential. Then 
\begin{enumerate}[(a)]
\item there exists a $(\phi-P(\phi))$-conformal measure $m_\phi$;
\item there exists a unique equilibrium measure $\mu_\phi$ for $\phi$;
\item there exists a positive 
    function $h_\phi\in L^1(m_\phi)$ such that $L_\phi h_\phi=e^{P(\phi)}h_\phi$ and $\mu_\phi=h_\phi m_\phi$;
\item For every $\psi \in C(\Sigma)$ we have
 $$\lim_{n \to \infty} \left\| e^{-nP(\phi)}  L^n_{\phi} (\psi) - \left( \int \psi \ d \mu_{\phi}  \right) h\right\|_{\infty} =0.$$

\item $m_\phi$ and $\mu_\phi$ are Gibbs measures;
\item the pressure function $t \mapsto P(t\phi)$ is real analytic on $\R$.
\end{enumerate}
\end{teo}
Part (c) of this theorem implies that these systems are positive recurrent according to our definition 
(this is also the case according to Sarig's definition). Notice also that part (f) of this theorem implies that such systems have no phase transitions.
On the other hand, Hofbauer \cite{Hnonuni} showed that for a particular class of non-H\"older potentials $\phi$, there are phase transitions, and also equilibrium states need not be unique. We describe this example in Section~\ref{ssec:ho-ke}.  We are led to the following natural question:

\textbf{Question:} how much can we relax the regularity assumption on the potential and still have uniqueness of the equilibrium measure/recurrence?

In order to give a partial answer to this question, Walters \cite{wa1} introduced the following class of functions.

\begin{defi}
For $\phi:\Sigma \to \R$, 
we say that $\phi:\Sigma \to \R$ is a \emph{Walters function} if for every $p \in \N$ we have $\sup_{n \ge 1} V_{n+p} (S_n\phi)< \infty$ and
 $$\lim_{p \to \infty} \sup_{n \ge 1} V_{n+p} (S_n\phi)=0.$$
We say that $\phi: \Sigma\mapsto \R$ is a \emph{Bowen function} if $$ \sup_{n \ge 1} V_n(S_n\phi) < \infty.$$
\end{defi}
Note that if $\phi$ is of summable variations then it is a Walters function. Walters showed that if a potential $\phi$ is Walters then it satisfies the Ruelle-Perron-Frobenius Theorem. In particular it has a unique equilibrium measure. Bowen introduced the class of functions we call Bowen in \cite{bo2}.  Note that every Walters function is a Bowen function and that there exist Bowen functions which are not Walters \cite{wa3}.  Bowen functions satisfy conditions (a)-(e) of Theorem~\ref{thm:Ru}, but not necessarily (f).
The following result was proved by Bowen \cite{bo1} and Walters \cite{wa2} and shows that for $(\Sigma, \sigma)$ a finite Markov shift and $\phi$ is a Bowen function, the system is recurrent.
\begin{teo}[Bowen-Walters]
For $(\Sigma, \sigma)$ a finite Markov shift, if $\phi: \Sigma \to \R$ is a Bowen function then there exists a unique equilibrium measure $\mu$ for $\phi$. Moreover, there exists a conservative $(\phi-P(\phi))$-conformal measure and the measure $\mu$ is exact.
\end{teo}

Bowen \cite{bo1} showed the existence of a unique equilibrium measure and Walters \cite{wa2} described the convergence properties of the Ruelle operator of a Bowen function. Recently, Walters \cite{wa3} defined a new class of functions that he called  `Ruelle functions' which includes potentials having more than one equilibrium measure.   

\subsection{Non-compact case} \label{ssec:nonc}

The definition of pressure in the case that the alphabet $S$ is finite (compact case) was introduced by Ruelle \cite{ru1}. In the (non-compact) case when the alphabet $S$ is infinite the situation is more complicated because the definition of pressure using $(n, \epsilon)-$ separated sets depends upon the metric  and can be different even for two equivalent metrics. Mauldin and Urba\'nski \cite{MUifs} gave a definition of pressure for  shifts on countable alphabets satisfying certain combinatorial assumptions.  Later, Sarig \cite{Sartherm}, generalising previous work by Gurevich \cite{Gushiftent, Gutopent}, gave a definition of pressure that satisfies the Variational Principle  for any topologically mixing countable Markov shift. This definition and the one given by Mauldin and Urba\'nski coincide for systems where both are defined.

Let $(\Sigma, \sigma)$ be a topologically mixing \emph{countable} Markov shift. This is a non-compact space.  Fix a symbol $i_0$ in the alphabet $S$ and
let $\phi \colon \Sigma \to \R$ be a Walters potential.  We refer to the corresponding cylinder $C_{i_0}$ as the \emph{base set} and let
\begin{equation}
Z_n(\phi, C_{i_0}):=\sum_{x:\sigma^{n}x=x} e^{S_n\phi(x)} \mathbb{1}_{C_{i_{0}}}(x),
\label{eq:Zn}
\end{equation}
where $\mathbb{1}_{C_{i_{0}}}$ is the characteristic function of the
cylinder $C_{i_{0}} \subset \Sigma$.  Also, defining  \begin{equation*}
r_{C_{i_0}}(x) := \mathbb{1}_{C_{i_0}}(x) \inf \{ n \ge 1 : \sigma^{n}x \in C_{i_0} \},
\end{equation*}
we let
\begin{equation}
Z_n^*(\phi, C_{i_0}):=\sum_{x:\sigma^{n}x=x} e^{S_n\phi(x)} \mathbb{1}_{\{r_{C_{i_0}} \}}(x),
\label{eq:Zn*}
\end{equation}
The so-called \emph{Gurevich pressure} of $\phi$ is defined by
\[
 P_G(\phi) := \lim_{n \to
\infty} \frac{1}{n} \log Z_n(\phi, C_{i_0}).
\]
This limit is proved to exist by Sarig  \cite[Theorem 1]{Sartherm}.
Since $(\Sigma, \sigma)$ is topologically mixing, one can show that $P(\phi)$ does not depend on the base set.
This notion of pressure coincides with the usual definition of pressure when the alphabet $S$ is finite and also satisfies the Variational Principle (see \cite{Sartherm}), i.e.,
$$P_G(\phi)=P(\phi).$$


Sarig showed in \cite[Theorem 1]{Sarnull} that exactly three different kinds of behaviour are possible for a Walters potential\footnote{Actually, he considered potentials of summable variations but the proofs of his results need no changes if it is assumed that the potential is a Walters function, see \cite{sanotes}.} $\phi$ of finite Gurevich pressure.  We adopt his definitions of transience and recurrence for a moment:
\label{p:sym conds}
\begin{enumerate}
\item[I.] The potential $\phi$ is  \emph{recurrent} if
\begin{equation}
 \sum_{n \ge 1} e^{-nP(\phi)} Z_n(\phi, C_{i_0})
 = \infty. \label{eq:Sar rec}
 \end{equation}
Here there exists a conservative $(\phi-P(\phi))$-conformal measure $m$.  If, moreover
\begin{enumerate}
\item $ \sum_{n \ge 1} ne^{-nP(\phi)} Z_n^*(\phi, C_{i_0}) < \infty$ then there exists an equilibrium measure for $(\Sigma, \sigma, \phi)$ absolutely continuous with respect to $m$. This is the \emph{positive recurrent} case;
\item $ \sum_{n \ge 1} ne^{-nP(\phi)} Z_n^*(\phi, C_{i_0})= \infty$ then there is no finite equilibrium measure absolutely continuous with respect to $m$.  This is the \emph{null recurrent} case;
 \end{enumerate}
\item[II.] The potential $\phi$ is  \emph{transient} if
\[ \sum_{n \ge 1} e^{-nP(\phi)} Z_n(\phi, C_{i_0})
 < \infty.\]  In this case there is no conservative $(\phi-P(\phi))$-conformal measure.  
\end{enumerate}

Cases I(a) and I(b) fit our definition (Definition~\ref{def:trans_new}) of positive and null recurrence respectively, and Case II fits our definition of transience.

%

\begin{rem}
If a potential $\phi$ is transient then it either has no conformal measure or  a dissipative conformal measure. Examples of both cases have been constructed by Cyr \cite[Section 5]{Cyr_thes}.  Moreover, examples are also given where there is more than one $\phi$-conformal measure in the transient setting.
\end{rem}

Recently Cyr and Sarig \cite{CyrSar} gave a characterisation of transient potentials which involves a phase transition of some pressure function, indeed they proved:


\begin{prop}[Cyr and Sarig] \label{prop:CyrSar}
The potential $\phi: \Sigma \to \R$ is transient if and only if for each $i \in S$ there exists $t_0 \in \R$ such that $P(\phi + t 1_{C_i}) = P(\phi)$ for every $t \le t_0$ and
$P(\phi + t 1_{C_i}) > P(\phi)$ for $t > t_0$.
\end{prop}
Moreover, Cyr \cite{Cyr_CMS} proved that, in a precise sense, most countable Markov shifts have at least one transient potential.  We could, in principle, use this as our definition of transience, but outside the domain of Markov shifts it has the disadvantages that it is unclear what kind of set should replace $C_a$, and it would appear to be very hard to check in any case.   Also the link between the conditions in Proposition~\ref{prop:CyrSar} and measures has not been established outside the Markov shift setting, which makes it difficult to interpret the condition in an ergodic theory context.

We conclude this section with a very important example of a countable Markov shift, the so called \emph{renewal shift}.
Let $S=\{0,1,2, \dots \}$ be a countable alphabet. Consider the
transition matrix $A=(a_{ij})_{i,j \in S}$ with $a_{0,0}= a_{0,n}=
a_{n,n-1}=1$ for each $n \ge 1$ and with all other entries equal to
zero. The \emph{renewal shift} is the (countable) Markov shift
$(\Sigma_R, \sigma)$ defined by the transition matrix $A$, that~is,
the shift map $\sigma$ on the space
\[
\Sigma_R := \left\{ (x_i)_{i \ge 0} : x_i \in S \text{ and } a_{x_i
x_{i+1}}=1 \text{ for each } i \ge 0\right\}.
\]
The \emph{induced system} $(\Sigma_I,\sigma)$ is defined as the
full-shift on the new alphabet given by
$\{ C_{0n(n-1)(n-2) \cdots 1}: n \ge
1 \}$. 
Given a function $\phi \colon \Sigma_R \to\R $ with summable
variation we define a new function, the \emph{induced potential},  $\Phi \colon \Sigma_R
\to\R$ by 
\begin{equation}
\Phi(x) := \sum_{k=0}^{r_{C_{0}}(x)-1}\phi(\sigma^{k} x),
\label{eq:ind pot}
\end{equation}
where the first return map $r_{C_{0}}$ is defined as above.
Sarig \cite{Sarphase} proved that if $\phi: \Sigma_R \to \R$ is a potential of summable variations, bounded above, with finite pressure and such that the induced potential
$\Phi$ is weakly H\"older continuous then there exists $t_c >0$ such that
\begin{equation*}
P(t\phi)=
\begin{cases}
\textrm{strictly convex and real analytic } & \textrm{ if } t \in [0, t_c),\\
At & \textrm{ if } t >t_c,
\end{cases}
\end{equation*}
where $A= \sup \{ \int \phi \ d \mu : \mu \in \mathcal{M} \}$.
This result is important since several of the examples known to exhibit phase transitions can be modelled by the renewal shift. Indeed, this is the case for the interval examples discussed in  Sections~\ref{ssec:ho-ke}--\ref{ssec:ma-po}.

\subsection{Non topologically mixing systems} \label{ssec:non-trans}
All the results we have discussed so far are under the assumption that the systems are topologically mixing. This is a standard irreducibility  hypothesis.  As we show below, it is easy to construct counterexamples to all the previous theorems when there is no mixing assumption.

Consider the dynamical system $(\Sigma_{0,1} \sqcup \Sigma_{2,3}, \sigma)$, where $\Sigma_{i,j}$ is the full-shift on the alphabet $\{i,j\}$. It is easy to see that the topological entropy of this system is equal to $\log 2$. Moreover, there exist two invariant measures of maximal entropy: the $(1/2, 1/2)$-Bernoulli measure supported in $\Sigma_{0,1}$ and
the $(1/2, 1/2)$-Bernoulli measure supported in $\Sigma_{2,3}$.  Therefore, the constant (and hence H\"older) potential $\phi(x)=0$ has two equilibrium measures. Actually, it is possible to construct a locally constant  potential exhibiting phase transitions. Let
\begin{equation*}
\psi(x)=
\begin{cases}
-1 & \textrm{ if } x \in \Sigma_{0,1},\\
-2 & \textrm{ if } x \in \Sigma_{2,3}.
\end{cases}\end{equation*}
The pressure function has the following form
\begin{equation*}
p_\psi(t)=
\begin{cases}
-t + \log 2 & \textrm{ if } t \ge 0;\\
-2t + \log 2 & \textrm{ if } t < 0.
\end{cases}\end{equation*}
Therefore the pressure exhibits a phase transition at $t=0$.  For $t> 0$ the equilibrium state for $t\psi$ is the
$(1/2, 1/2)$-Bernoulli measure supported on $\Sigma_{0,1}$ and for $t<0$ the equilibrium state for $t\psi$ is the
$(1/2, 1/2)$-Bernoulli measure supported on $\Sigma_{2,3}$.  For $t=0$ these measures are both equilibrium states.  Note that in both cases these measures are also $(t\psi-p_\psi(t))$-conformal, so the phase transitions here are not linked to transience.

Phase transitions caused by the  non mixing structure of the system also appear in the case of interval maps. Indeed, the renormalisable examples studied by Dobbs \cite{Dobphase} are examples of this type.

\section{The interval case}
\label{sec:intervals}

In this section we describe examples of systems of interval maps and potentials with phase transitions and explain how our definition of recurrence/transience is an improvement on alternative notions there.

The situation in the compact interval context is different from that of the compact symbolic case in that   rather smooth potentials can have more than one equilibrium measure.  All the examples we consider are such that entropy map is upper semi-continuous.  Since the interval is compact, weak$^*$ compactness of the space of invariant probability measures implies that every continuous potential has (at least) one equilibrium measure. The study of phase transitions in the context of topologically mixing interval maps is far less developed that in the case of Markov shifts. 
We review some of these examples.

\subsection{Hofbauer-Keller} \label{ssec:ho-ke}
The following example was constructed by Hofbauer and Keller \cite{HofKel_pw} based on previous work in the symbolic setting by Hofbauer \cite{Hnonuni}. We will present it defined in a half open interval, but one can also think of this as a dynamical system on a compact set, namely the circle.

The dynamical system considered is the angle doubling map $f: [0,1) \mapsto [0,1)$  defined  by $f(x)=2x\ (\text{mod } 1)$.  Typically this map is studied via its relation to the full shift on two symbols, so the continuity of potentials is only required for the symbolic version of the potential.   Given a sequence of real numbers $(a_k)_{k\in \N_0}$ such that $\lim_{k\to \infty}a_k=0$, we define the potential $\phi$ by
\begin{equation*}
\phi(x)= \begin{cases} a_k  & \text{ if } x\in [2^{-k-1},2^{-k}),\\
0 & \text{ if } x=0.
\end{cases}
\end{equation*}
Notice that this potential is continuous on $[0,1)$ equipped with the metric induced by the standard metric on the full shift on two symbols.

Let $F$ be the first return map to $X=[1/2, 1)$ with return time $\tau$.  So for $X_n:=\{\tau\}$, the induced potential $\Phi$ (see \eqref{eq:ind pot}) takes the value $$s_n:= \sum_{k=0}^{n-1} a_k.$$
Figure \ref{table:hof} summarises the possible behaviours of the thermodynamic formalism depending on the sums $s_n$. Note that there was mistake\footnote{The third entrance in the column of \emph{Gibbs measures} was \emph{no} in Hofbauer's \cite{Hnonuni} and it should have been  \emph{yes}.} in a similar table in the original paper \cite{Hnonuni}, corrected by Walters in \cite[p.1329]{wa2}.  In the final column we apply our definition of recurrence/transience.  The first four entries in that column follow directly from results of \cite{Hnonuni} while the final entry follows from Lemma~\ref{lem:Hof trans} below.


\begin{figure}[ht]
\unitlength=5mm
\begin{picture}(20,12)(3,0)
\thinlines
\put(0,0){\line(0,1){8}}
\put(0,0){\line(1,0){26}}
\put(0.2,0.8){$\sum_n e^{s_n} < 1$}
\put(0,2){\line(1,0){26}}
\put(0.2,3.3){$\sum_n e^{s_n} = 1$}
\put(0,5){\line(1,0){26}}
\put(0.2,6.3){$\sum_n e^{s_n} > 1$}
\put(0,8){\line(1,0){26}}

\put(3.9,0){\line(0,1){8}}
\put(3.9,3.5){\line(1,0){22.1}}
\put(3.9,6.5){\line(1,0){22.1}}


\put(4.4,2.5){\small $\sum_n (n+1)e^{s_n} = \infty$}
\put(4.4,4){\small $\sum_n (n+1)e^{s_n} < \infty$}
\put(4.4,5.5){\small $\sum_k a_k = \infty$}
\put(4.4,7){\small $\sum_k a_k < \infty$}

\put(10,0){\line(0,1){12}}
\put(13.5,0){\line(0,1){12}}
\put(16.8,0){\line(0,1){12}}
\put(21.2,0){\line(0,1){12}}
\put(26,0){\line(0,1){12}}

\put(10.4,11){\small }
\put(10.4,10){\small Pressure}
\put(10.4,9){\small $P(\phi)$}

\put(10.3,7){$P(\phi) > 0$}
\put(10.3,5.5){$P(\phi) > 0$}
\put(10.3,4){$P(\phi) = 0$}
\put(10.3,2.5){$P(\phi) = 0$}
\put(10.3,1){$P(\phi) = 0$}

\put(14,11){\small $\mu_\phi$ is}
\put(14,10){\small a Gibbs}
\put(14,9){\small measure}

\put(14.6,7){\small yes}
\put(14.6,5.5){\small no}
\put(14.6,4){\small yes}
\put(14.6,2.5){\small no}
\put(14.6,1){\small no}

\put(17.3,11){\small $\phi$ has a}
\put(17.3,10){\small unique equi-}
\put(17.3,9){\small librium state}

\put(18.4,7){\small yes}
\put(18.4,5.5){\small yes}
\put(18.4,4){\small no}
\put(18.4,2.5){\small yes}
\put(18.4,1){\small yes}

\put(21.4,11){\small $\phi$ is }
\put(21.4,10){\small $+$ve recurrent/}
\put(21.4,9){\small transient}

\put(21.6,7){\small $+$ve recurrent}
\put(21.6,5.5){\small $+$ve recurrent}
\put(21.6,4){\small $+$ve recurrent}
\put(21.6,2.5){\small null recurrent}
\put(21.6,1){\small transient}

\put(10,12){\line(1,0){16}}
\end{picture}
\caption{The first two columns summarise results in \cite{Hnonuni}: Equation (2.6) and Section 5.  The final column applies Definition~\ref{def:trans_new}}
\label{table:hof}
\end{figure}

We can make choices of $(a_n)_n$ so that the pressure function has the form: \begin{equation}
p_\phi(t)=
\begin{cases}
\textrm{strictly convex and real analytic } & \textrm{ if } t \in [0, 1),\\
0 & \textrm{ if } t >1.
\end{cases}\label{eq:Hof graph}
\end{equation}
The pressure is not analytic at $t=1$.  The real analyticity of the pressure follows, for instance, from \cite{Sarphase}. The general strategy to prove this type of result is to prove that the transfer operator is quasicompact. That is, to show that the essential spectral radius is strictly smaller  than the spectral radius. This means that except for the spectrum inside a small disc the operator behaves like a compact operator (where the spectrum consists of isolated eigenvalues of finite multiplicity). Since the leading eigenvalue corresponds to the exponential of the pressure function, classic perturbation arguments allow for the proof of real analyticity of the pressure.

Moreover, we can choose $(a_k)_k$ so that:
\begin{itemize}
\item the map $t\mapsto P(t\phi)$ is differentiable at $t=1$ where $\phi$ has  only one equilibrium measure (the Dirac delta at zero).  This is the fourth row in Figure~\ref{table:hof};
\item  the map $t\mapsto P(t\phi)$ has a first order phase transition at $t=1$, i.e., is not differentiable at $t=1$.  Here and $\phi$ has two equilibrium states, one is the Dirac delta at zero and the other  can be seen as the projection of the Gibbs measure $\mu_{\Phi}$, the equilibrium state for $\Phi$. This is the third row in Figure~\ref{table:hof}.
\end{itemize}

Following \cite{HofKel_pw}, for $\alpha>0$ we could choose our sequence $(a_k)_k$ to be asymptotically like $\alpha \log \left( \frac{k+1}{k+2} \right)$ for all large $k$.  The particular choice of $\alpha$ separates the two cases above. 

  The following lemma shows that by our definition of recurrence, the phase transition here corresponds to a switch from recurrence to transience.

\begin{lema}
If $(a_k)_k$ are chosen so that $\sum_ne^{s_n}<1$, then $(X,f, \phi)$ is transient.  In fact, there exists a there exists a $(\phi-P(\phi))$-conformal measure, but any such measure is dissipative.
\label{lem:Hof trans}
\end{lema}

\begin{proof}
Since $\phi$ gives rise to a continuous potential for the full shift on two symbols, \cite{waTA} \textbf gives us a conformal measure $\nu$ such that $L_{\phi}^*\nu=e^{P(\phi)}\nu=\nu$, since $P(\phi)=0$.   In other words, $\nu$ is a $(\phi-P(\phi))$-conformal measure. Note that $\nu$ must be a finite measure, so we will assume that it is a probability measure. We will show that $\nu$ must be dissipative.

Conformality implies that for $n\ge1$, we have
$$1=\nu([0,1))=\nu(f^n(X_n))=\int_{X_n}e^{-s_n}~d\nu,$$
so $\nu(X_n)=e^{s_n}$ which implies
\begin{equation}1=\nu([0,1))=\nu(\{0\})+\sum_n e^{s_n}.\label{eq:noconf Hof}
\end{equation}
Therefore $\nu(\{0\})>0$.  Moreover, since we can similarly show that $\nu(\{1/2\})=\nu(\{0\})e^{a_0}>0$,  and since
$\left\{f^{-n}(1/2)\right\}_{n\ge 0}$ is a wandering set, it follows that $\nu$ is dissipative.
\end{proof}

We also note that in the above proof,  for the induced potential $\Phi$, we can show that $P(\Phi)<0$, from which we can give an alternative proof that any $\phi$-conformal measure must be supported on $\left\{f^{-n}(0)\right\}_{n\ge 0}$ using the techniques in Sections~\ref{sec:no conf}  and \ref{sec:ex}.

Now let us compare our definition of recurrence with \eqref{eq:Sar rec}.  Let us suppose that $\sum_ne^{s_n}<1$, so by the lemma above, our system is transient.   The pressure here is clearly zero, so it would only make sense to put this value into the computation of recurrence in \eqref{eq:Sar rec}.   If we choose our base set to be $[1/2,1)$, then the definition of transience in \eqref{eq:Sar rec} fits with ours.  However, if we choose our base set to be, for example, $A:=[0,1/2)$, then \begin{equation}
Z_n(\phi, A)\ge e^{-nP(\phi)}e^{S_n\phi(0)}=1,
\label{eq:HoKe bad def}
\end{equation}
 which suggests that the system is actually recurrent.  So blindly applying the computation of \eqref{eq:Sar rec} here leads to recurrence being ill-defined.  We note that the same kind of argument can be made against the application of Proposition~\ref{prop:CyrSar}.

\subsection{Manneville-Pomeau} \label{ssec:ma-po}
The following example was introduced by Manneville and Pomeau in \cite{ManPom}. It is one of the simplest examples of a non-uniformly hyperbolic map. It is expanding and it has a parabolic fixed point at $x=0$.   For these systems and for the type of potential we choose below, there are the same issues with the definition of recurrence as in Section~\ref{ssec:ho-ke}.

We give the form studied in \cite{LivSaVai}.  For $\alpha >0$, the map is defined by
\begin{equation} f(x)=\begin{cases} x(1+2^\alpha x^\alpha) & \text{if } x \in [0, 1/2),\\
2x-1 & \text{if } x \in [1/2, 1).\end{cases}
\end{equation}

The pressure function of the potential $-\log |f'|$ has the following form (see, for example, \cite{Sarphase}),
\begin{equation*}
p(t)=
\begin{cases}
\textrm{strictly convex and real analytic } & \textrm{ if } t \in [0, 1),\\
0 & \textrm{ if } t >1,
\end{cases}
\end{equation*}
where, for brevity we let
$$p(t):=P(-t \log |f'|).$$
(We use this notation throughout for this particular kind of potential.)
If $\alpha \in (0,1)$  then the map $f$ has a probability invariant measure absolutely continuous with respect to the Lebesgue measure, which is also an equilibrium state for $-\log|Df|$, and the pressure function is not differentiable at $t=1$.  This setting is analogous to that in the penultimate row of Table~\ref{table:hof}.  On the other hand, if $\alpha \ge 1$ then there is no absolutely continuous invariant probability measure and the pressure function is differentiable at $t=1$.  This setting is analogous to that in the last row of Table~\ref{table:hof}.

The value of $\alpha$ determines the class of differentiability of the map $f$ and determines the amount of time `typical' orbits spend near the parabolic fixed point.
For $\alpha \in (0,1)$, the amount of time spent near the point $x=0$ by Lebesgue-typical points is not long enough to force the relevant invariant measure (the equilibrium state for $-\log|Df|$) to be infinite. But if  $\alpha \ge 1$ then the map has a sigma-finite (but infinite) invariant measure absolutely continuous with respect to the Lebesgue measure.

As in Lemma \ref{lem:Hof trans}, our definition of recurrence gives that for all $\alpha>0$ and $t>1$, the system is transient.  Again, as in the argument around \eqref{eq:HoKe bad def}, our definition of recurrence/transience is more applicable here than \eqref{eq:Sar rec}.

\begin{rem}
Note that the Dirac delta measure on $0$ is a  conformal measure for
$-t\log|Df|$ with $t>1$ if we remove all preimages of 0. 
\end{rem}

%

\subsection{Chebyshev} \label{ssec:cheby}
A simple example of a transitive map in the quadratic family which exhibits a phase transition
is the Chebyshev polynomial $f(x):= 4x(1-x)$ defined on $[0,1]$ (see  for example \cite{Dobphase}).  
For the set of \emph{geometric potentials} $\{-t\log|f'|:t\in \R\}$, if $t>-1$ then the equilibrium state for $-t\log|f'|$ is the absolutely continuous (with respect to Lebesgue) invariant probability measure $\mu_1$, which has $\int \log|f'| \ d \mu_1=\log 2$.  If $t<-1$ then the equilibrium state for $-t\log|f'|$ is the Dirac measure $\delta_0$ on the fixed point at 0, which has $\int \log |f'| \ d\delta_0= \log 4$. 
 So, there exists a phase transition at $t_0=-1$ and
\begin{equation*}
p(t)= \begin{cases} -t\log 4&  \text{if } t < -1,\\
 (1-t)\log 2& \text{if } t\ge -1.
\end{cases}
\end{equation*}

For $t<-1$, if our base set $A$ includes the fixed point 0, then as in \eqref{eq:HoKe bad def}, we obtain $Z_n(\phi, A)\ge 1$, which indicates recurrence.  However, this situation should clearly not be thought of as recurrent.  Indeed, the arguments in Section~\ref{sec:no conf}  can be adapted to show that this is transient by our definition.

\subsection{Multimodal maps} \label{ssec:multi}
Up until now, our examples have had `bad' potentials, but the underlying dynamical system has nevertheless had some Markov structure.  In this section we introduce a standard class of maps, many of which have no such structure.  We study this class in more depth in Section~\ref{sec:no conf}.

Let $\mathcal F$ be the collection of $C^2$ multimodal interval maps $f:I \to I$, where $I=[0,1]$, satisfying:

\newcounter{Lcount}
\begin{list}{\alph{Lcount})}
{\usecounter{Lcount} \itemsep 1.0mm \topsep 0.0mm \leftmargin=7mm}
\item the critical set $\crit = \crit(f)$ consists of finitely many critical points $c$ with critical order $1 < \ell_c < \infty$, i.e., there exists a neighbourhood $U_c$ of $c$ and a $C^2$ diffeomorphism $g_c:U_c \to g_c(U_c)$ with $g_c(c) = 0$
     $f(x) = f(c) \pm |g_c(x)|^{\ell_c}$;
\item $f$ has negative Schwarzian derivative, i.e., $1/\sqrt{|Df|}$ is convex;
\item $f$ is topologically transitive on $I$;
\item $f^n(\crit)  \cap f^m(\crit)=\es$ for $m \neq n$.
\end{list}

For $f\in \F$ and $\mu\in \M_f$, let us define,
$$\lambda(\mu):= \int \log |f'| \ d \mu \qquad \text{ and } \qquad \lambda_m:= \inf \{ \lambda(\mu) : \mu \in \mathcal{M}_f \}.$$

It was proved in \cite{ITeq} that there exists $t^+>0$ such that the pressure function of the (discontinuous) potential $\log |f'|$ satisfies,
\begin{equation*}
p(t)=
\begin{cases}
\textrm{strictly convex and }C^1  & \textrm{ if } t \in (-\infty, t^+),\\
At & \textrm{ if } t >t^+.
\end{cases}
\end{equation*}
In the case $\lambda_m=0$, $t^+\le 1$ and $A=0$; while in the case $\lambda_m>0$, $t^+>1$ and $A<0$.

\begin{rem}
The number of equilibrium measures at the phase transition can be large. Indeed, Cortez and Rivera-Letelier \cite{CJ}  proved that given $\mathcal{E}$ a non-empty, compact, metrisable and totally disconnected topological space then there exists a parameter $\gamma \in (0,4]$ such that set of invariant probability measures of $x\mapsto \gamma x(1-x)$, supported on the omega-limit set of the critical point is homeomorphic to $\mathcal{E}$. Examples of quadratic maps having multiple ergodic measures supported on the omega-limit set of the critical point were  first constructed in \cite{Brminim}.
\end{rem}

Again \eqref{eq:Sar rec} will not give us a reasonable way to check recurrence/transience for these maps due to the poor smoothness properties of the potential as well as the lack of Markov structure.  
Clearly, given what happened in the previous examples, we would expect that the phase transition at $t^+$ marked the switch from recurrence to transience, and indeed we show this in Section~\ref{sec:no conf}.  However, as we show in Section~ \ref{sec:ex} (see also \cite[Example 3]{Sarphase}), it can happen that for systems $(X,f, t\phi)$ increasing the parameter $t$ can take us from recurrence through transience and then out to recurrence again.  So we should check the recurrence/transience of the systems in $\F$ outlined above.  This is done Section~\ref{sec:no conf}.

\subsection{Brief summary of recurrence for interval maps}

All the examples of phase transitions presented in this section have the same type of behaviour. That is, the pressure function has one of the following two forms:
\begin{equation}
p_\phi(t)=
\begin{cases}
\textrm{strictly convex and  differentiable} & \textrm{ if } t \in [0, t_0),\\
At & \textrm{ if } t >t_0,
\end{cases} \label{eq:conv pres}
\end{equation}
where $A \in \R$ is a constant. The regularity at the point $t=t_0$ varies depending on the examples. The other possibility is

\begin{equation}
p_\phi(t)=
\begin{cases}
Bt+C  & \textrm{ if } t \in [0, t_0),\\
At & \textrm{ if } t >t_0,
\end{cases} \label{eq:affine pres}
\end{equation}
where $A,B,C \in \R$ are constants.

\begin{rem}
It also possible for the pressure function to have the `reverse form' to the one given in equation \eqref{eq:affine pres}: i.e., there are interval maps and potentials for which the pressure function has the form $p_\phi(t)=At $ in an interval $(-\infty,t_0]$ and $p_\phi(t)=Bt+C $ for $t>t_0$. The same `reverse form' exists  in the case that the pressure function is given as in equation \eqref{eq:conv pres}.
\end{rem}

Essentially what happens is that the dynamics can be divided into an hyperbolic part and a non-hyperbolic part (the latter having zero entropy, for example a parabolic fixed point or the post-critical set).

\begin{rem}
 As in Section~\ref{ssec:non-trans},  the situation  can be completely different if the map is not assumed to be topologically mixing.
\end{rem}

A natural question that arises when considering the above examples is the following:
Must the onset of transience always give pressure functions of the type in \eqref{eq:conv pres} or \eqref{eq:affine pres} (i.e., the onset of transience occurs `at zero entropy' and once a potential is transient for some $t_0$ is either transient for all $t<t_0$ or $t>t_0$)?  This is shown to be false in Section~\ref{sec:ex}.  In the case of shift spaces we point out \cite{Olivier} as well as the non-constructive Example 3 in \cite{Sarphase}.

\section{No conservative conformal measure}
\label{sec:no conf}

In this section we will show that the systems considered in Section~\ref{sec:intervals} are transient past the phase transition. We focus on multimodal maps $f\in \F$ defined in Section~\ref{ssec:multi}.  We will show that for a certain range of values of $t \in \R$ the potential $-t \log|f'|$ has no conservative conformal measure and hence is transient. The results described here also hold for the Manneville-Pomeau map, but since the proof is essentially the same, but simpler, we only discuss the former case. The first result deals with the recurrent case.

\begin{teo}
Suppose that $f\in \F$.  If $t<t^+$ then there is a
weakly expanding conservative $(-t\log|Df|-p(t))$-conformal measure.
\end{teo}

This is proved in the appendix of \cite{T}, where it is referred to as Proposition 7'.  The idea of the proof is to obtain a conformal measure for an inducing scheme, as described below, and then spread this measure around the space in a canonical way.

\begin{prop}
Suppose that $f:I \to I$ belongs to $\F$ and $\lambda_m=0$.  Then for any $t>1$, $(I, f, -t\log|Df|)$ is transient.
\label{prop:mult trans}
\end{prop}

This proposition covers the case when $t^+=1$.  We expect this to also hold when $t^+\neq 1$, but we do not prove this.
As in Sections~\ref{ssec:ho-ke} and \ref{ssec:ma-po}, the strategy used to study multimodal maps $f \in \mathcal{F}$,  and indeed to prove Proposition~\ref{prop:mult trans}, considering that they lack Markov structure and  uniform expansivity, is to consider a generalisation of the first return map. These maps are expanding and are Markov (although over a countable alphabet). The idea is to study the `inducing scheme'  through the theory of Countable Markov Shifts and then to translate the results into the original system. 

We say that $\left(X,\{X_i\}_i, F,\tau\right)=(X,F,\tau)$ is an \emph{inducing scheme} for $(I,f)$ if
\begin{list}{$\bullet$}{\itemsep 0.2mm \topsep 0.2mm \itemindent -0mm \leftmargin=5mm}
\item $X$ is an interval containing a finite or countable
collection of disjoint intervals $X_i$ \st $F$ maps each $X_i$
diffeomorphically onto $X$, with bounded distortion (i.e. there
exists $K>0$ so that for all $i$ and $x,y\in X_i$, $1/K\le DF(x)/DF(y) \le K$);
\item $\tau|_{X_i} = \tau_i$ for some $\tau_i \in \N$ and $F|_{X_i} = f^{\tau_i}$.  If $x \notin \cup_iX_i$ then $\tau(x)=\infty$.
\end{list}
The function $\tau:\cup_i X_i \to \N$ is called the {\em inducing time}. It may
happen that $\tau(x)$ is the first return time of $x$ to $X$, but
that is certainly not the general case.   We denote the set of points $x\in I$ for which there exists $k\in \N$ such that $\tau(F^n(f^k(x)))<\infty$ for all $n\in \N$ by $(X,F,\tau)^\infty$.

The space of $F-$invariant measures is related to the space of $f$-invariant measures. Indeed, given an $f$-invariant measure $\mu$, if there is an $F$-invariant measure $\mu_F$ such that for a subset $A\subset I$,
\begin{equation} \label{eq:ind-meas}
\mu(A)= \frac{1}{\int\tau~d\mu_F}\sum_{k=1}^{\infty} \sum_{i=0}^{k-1} \mu_F \left( f^{-k}(A) \cap X_i \right)
\end{equation}
where $\frac{1}{\int\tau~d\mu_F}<\infty$,
we call $\mu_F$ the \emph{lift} of $\mu$ and say that $\mu$ is a \emph{liftable} measure. Conversely, given a measure $\mu_F$ that is $F$-invariant we say that $\mu_F$ \emph{projects} to $\mu$ if \eqref{eq:ind-meas} holds. We say that a $f-$invariant probability measure $\mu$ is \emph{compatible} with the inducing scheme $(X,F, \tau)$ if
\begin{itemize}
\item We have that $\mu(X)>0$ and $\mu(X \setminus (X, F)^{\infty})=0$, and
\item there exists a $F-$invariant measure $\mu_F$ which projects to $\mu$
\end{itemize}

\begin{rem} \label{rmk:Abram}
Let $\mu$ be a liftable measure and  be $\nu$ be its lift. A classical result by Abramov \cite{Abra} (see also \cite{PeSe,zw})  allow us to relate the entropy of both measures. Further results obtained in  \cite{PeSe,zw} allow us to do the same with the integral of a given potential  $\phi:I\to \R$.   Indeed, for the induced potential $\Phi$ we have that
\begin{equation*}
h(\mu)=\frac{h(\nu)}{\int \tau \ d \nu}  \textrm{ and  } \int \phi \ d \mu= \frac{\int \Phi \ d \nu}{\int \tau \ d \nu}.
\end{equation*}
Also a $\phi$-conformal measure $m_\phi$ for $(I,f)$ is also a $\Phi$-conformal measure for $(X,F)$ if $m_\phi(\cup X_i)=m_\phi(X)$.
\end{rem}

For $f\in \F$ we choose the domains $X$ to be $n$-cylinders coming from the so-called \emph{branch partition}: the set $\P_1^f$ consisting of maximal intervals on which $f$ is monotone.  So if two domains $C_1^i, C_1^j \in \P_1^f$ intersect, they do so only at elements of $\crit$. The set of corresponding $n$-cylinders is denoted $\P_n^f:=\vee_{k=1}^nf^{-k}\P_1$.  We let $\P_0^f:=\{I\}$.  For an inducing scheme $(X,F,\tau)$ we use the same notation for the corresponding $n$-cylinders $\P_n^F$.

The following result, proved in \cite{T} (see also \cite{BTeqnat, ITeq})  proves that useful inducing schemes exist for maps $f \in \mathcal{F}$.

\begin{teo}\label{thm:schemes}
Let $f\in \F$.  There exist a countable collection $\{(X^n,F_n,\tau_n)\}_n$ of inducing schemes with $\bd X^n \notin (X^n,F_n,\tau_n)^\infty$ such that any ergodic invariant probability measure $\mu$ with $\lambda(\mu)>0$ is compatible with one of the inducing schemes $(X^n, F_n,\tau_n)$.
Moreover, for each $\eps>0$ there exists $N\in \N$ such that
$LS_\eps\subset \cup_{n=1}^N(X^n, F_n, \tau_n)^\infty$.
\end{teo}

We are now ready to apply this theory to the question of transience, building up to proving Proposition~\ref{prop:mult trans}.

\begin{lema}
Suppose that $f\in\F$.  If, for $t >t^+$, there is a conservative weakly expanding $(-t\log|Df|-s)$-conformal measure $m_{t,s}$ for some $s\in \R$, then $s\le P(-t\log|Df|)$.  Moreover, there is an inducing scheme  $(X,F,\tau)$ such that $$P(-t\log|DF|-\tau s)= 0$$ and
$$m_{t,s}\left(\left\{x\in X:\tau^k(x) \text{ is defined for all } k\ge 0\right\}\right)=m_{t,s}(X).$$
\label{lem:s=P}
\end{lema}

\begin{proof}
We prove the second part of the lemma first.

Suppose that $m_{t,s}$ is a weakly expanding $(-t\log|Df|-s)$-conformal measure.  We introduce an inducing scheme $(X,F)$.
Since $m_{t,s}$ is weakly expanding, by Theorem~\ref{thm:schemes} there exists an inducing scheme $(X,F, \tau)$ such that \begin{equation}
m_{t,s}\left(\left\{x\in X:\tau^k(x) \text{ is defined for all } k\ge 0\right\}\right)=m_{t,s}(X)>0.
\label{eq:ind time meas}
\end{equation}

This can be seen as follows.   Theorem~\ref{thm:schemes} implies that there exists $(X,F)$ with $m_{t,s}((X,F)^\infty)>0$.  If $m_{t,s}(X\sm(X,F)^\infty)>0$ then there must exist $k\in \N$ such that the set 
$$A_k:=\left\{x\in X:\exists \ i(x)\in \N \text{ such that } \tau^{i(x)}(x)=k \text{ but } \tau(F^{i(x)})=\infty\right\}$$
has $m_{t,s}(A_k)>0$.  But a standard argument shows that $A_k$ is a wandering set: if not then there exist $n>j\in \N_0$ such that there is a point $x\in f^{-n}(A_k)\cap f^{-j}(A_k)$: so $z:=f^j(x)\in A_k$ and $f^{n-j}(z)\in A_k$, which is impossible.

By the distortion control for the inducing scheme, for any $n$-cylinder  $\c_{n,i}\in \P_n^F$  and since $m_{t,s}(X)=\int_{\c_{n,i}}|DF^n|^te^{s\tau^n}~dm_t$, there exists $K\ge 1$ such that
\begin{equation}
|\c_{n,i}|^te^{s\tau^n} =K^{\pm t}|X|^t m_{t,s}(\c_{n,i}).
\label{eq:cyl meas}
\end{equation}

Since the inducing scheme is the full shift, and because of this distortion property, the pressure of $-t\log|DF|-s\tau$ can be computed as
$$\lim_{n\to \infty}\frac{\log\left(\sum_{\c_{n,i}\in \P_n^F}|\c_{n,i}|^te^{s\tau^n}\right)}n.$$
However, using first \eqref{eq:cyl meas} and then \eqref{eq:ind time meas}, we have
$$\sum_{\c_{n,i}\in \P_n^F}|\c_{n,i}|^te^{s\tau^n}= K^{\pm t}|X|^t \sum_{\c_{n,i}\in \P_n^F} m_{t,s}(\c_{n,i})=K^{\pm t}|X|^t  m_{t,s}(X)$$
for all $n\ge 1$.
This implies that $P(-t\log|DF|-s\tau)=0$, proving the second part of the lemma.

We prove the first part of the lemma by applying the Variational Principle to the inducing scheme.  Since $P(-t\log|DF|-s\tau)=0$, by \cite[Theorem 2]{Sartherm}, there exists a sequence $(\mu_{F,n})_n$ each supported on a finite number of cylinders in $\P_1^F$ and with
$$\lim_{n \to \infty} \left( h(\mu_{F,n})+\int -t\log|DF|-s \int \tau~d\mu_{F,n} \right) = 0.$$
Therefore, by the Abramov Theorem (see Remark~\ref{rmk:Abram}), for the projected measures $\mu_n$ we have
$$h(\mu_{n})-\int t\log|Df|~d\mu_{n}\to s.$$
Hence the definition of pressure implies that $s\le p(t)$.
\end{proof}

\begin{proof}[Proof of Proposition~\ref{prop:mult trans}.]
Suppose that there exists a weakly expanding conservative $-t\log|Df|$-conformal measure $m_t$.  Let $(X,F)$ be the inducing scheme in Lemma~\ref{lem:s=P}, with distortion $K\ge 1$.  Then $P(\Psi_t)=0$ and 
\begin{align*}
m_t(X) & =\sum_{\c_{n,i}\in \P_n^F} m_t(\c_{n,i}) = K^{\pm t} |X|^t\sum_{\c_{n,i}\in \P_n^F}|\c_{n,i}|^t\\
 &= K^{\pm t} |X|^t \sum_{\c_{n,i}\in \P_n^F}|\c_{n,i}|^{}|\c_{n,i}|^{t-1}\\
 &\le K^{t}|X|^t \left(\sup_{\c_{n,i}\in \P_n^F}|\c_{n,i}|\right)^{t-1}\sum_{\c_{n,i}\in \P_n^F}|\c_{n,i}|\\
 &\le K^{t}|X|^{t+1} \left(\sup_{\c_{n,i}\in \P_n^F}|\c_{n,i}|\right)^{t-1}.
 \end{align*}
Since $t>1$ by choosing $n$ large, we can make this arbitrarily small, so we are led to a contradiction.
\end{proof}

\section{Possible transient behaviours} \label{sec:ex}

In this section we address some of the questions  raised about the possible behaviours of transient systems in Section~\ref{sec:trans}.  In particular,
we present an example which gives us a range of possible behaviours for a pressure function which has one or two phase transitions.  This example is very similar to that presented by Olivier in
\cite[Section 4]{Olivier} in which he extended the ideas of Hofbauer \cite{Hnonuni} to produce a system with hyperbolic dynamics, but with a potential $\phi$ which was sufficiently irregular to produce a phase transition: the support of the relevant equilibrium states $t\phi$ jumping from the whole space to an invariant subset as $t$ moved through the phase transition.  We follow the same kind of argument, with slightly simpler potentials.   In our case, we are able to obtain very precise information on the pressure function and on the measures at the phase transition.  Moreover, we can arrange our system so that the support of the relevant equilibrium state for $t\phi$ jumps from the whole space, to an invariant subset, and then back out to the whole space as $t$ increases from $-\infty$ to $\infty$.  Between the phase transitions we have transience.
We point out that Sarig proved the existence of such phenomena in \cite{Sarphase}, but here we are able to give an explicit, and fairly elementary, construction.

\begin{defi}
For a dynamical system $(X,f)$ with a potential $\phi$, let us consider conditions i) $\lim_{t\to -\infty}p_\phi(t)= \infty$; ii) there exist $t_1<t_2$ such that $p_\phi(t)$ is constant on $[t_1, t_2]$; iii) $\lim_{t\to \infty}p_\phi(t) =\infty$.  We say that $p_\phi$ is \emph{DF} (for down-flat) if i) and ii) hold; that $p_\phi$ is \emph{DU} (for down-up) if i) and iii) hold; that $p_\phi$ is \emph{DFU} (for down-flat-up) if i), ii) and iii) hold.
\end{defi}

In this section we describe a situation with pressure which is DFU.
The system is the full-shift on three symbols $(\Sigma_3, \sigma)$.  (Note that we could instead consider $\{ I_i \}_{i=1}^{3}$, three pairwise disjoint intervals contained in $[0,1]$, and the map
$f:\bigcup_{i=1}^{3} I_i  \subset [0,1]\to [0,1]$, where $f(I_i)=[0,1]$ which is topologically (semi-)conjugated to $(\Sigma_3, \sigma)$.)
The construction we will use can be thought of as a generalisation of the renewal shift  (see Section~\ref{ssec:nonc}). Let $(\Sigma_3, \sigma)$ be the full shift on three symbols $\{1,2,3\}$.  A point $x \in \Sigma_3$ can be written as $x=(x_0x_1 x_2 \dots)$, where $x_i \in \{1,2,3\}.$ Our \emph{bad set} (the we will denote by $B$) will be the full shift on two symbols $\{1,3\}$ and the \emph{renewal vertex} will be $\{2\}$.

For $N\ge 1$ and $(x_0,x_1,\ldots, x_{N-1})\in \{1,2,3\}^N$, let $[x_0x_1\ldots x_{N-1}]$ denote the cylinder $C_{x_0x_1\ldots x_{N-1}}$.
We set $X_0$ to be the cylinder $[2]$ and define the \emph{first return time} on $X_0$ as the function $\tau:[2] \to \N$ defined by $\tau(x)= \inf \{n \in \N: \sigma^n x \in [2]\}$. 
We set $$X_n:=\left\{x\in [2]:\tau(x)\right\}.$$
This consists of $2^{n-1}$ cylinders. Indeed, we list the first three sets from which this assertion is already clear,
\[X_1= [22]\]
\[X_2= [212]\cup [232]\]
\[X_3= [2112]\cup [2132]\cup [2312]\cup [2332]. \]

The class of potentials is given as follows.

\begin{defi}
A function $\phi : \Sigma_3 \to \mathbb{R}$ is called a \emph{grid function} if it is of the form
\[ \phi(x) = \sum_{n=0}^{\infty} a_n \cdot \mathbb{1}_{X_n}(x), \]
where $\mathbb{1}_{X_n}(x)$ is the characteristic function of the set $X_n$ and $(a_n)_{n \in \mathbb{N}}$ is a sequence of real numbers such that $\lim_{n \to \infty} a_n =0$. Note that $\phi|_B=0$.
\end{defi}
\textbf{MT:added a bit.} Grid functions were introduced in a more general form by Markley and Paul \cite{MarkPaul}: they allowed the set $B$ to be any subshift and $X_n$ to be any partition elements converging to $B$ in the Hausdorff metric.  They were presented as a generalisation of those functions by Hofbauer \cite{Hnonuni} which we described in Section~\ref{ssec:ho-ke}. They can be thought of as weighted distance functions to a \emph{bad} set $B$.  Recently, this type of potentials were used to disprove an ergodic optimisation conjecture  \cite{chh}.

We are now ready to state our result concerning the thermodynamic formalism for our grid functions.

\begin{teo}
Let $(\Sigma_3, \sigma)$ be the full-shift on three symbols and let $\phi:\Sigma_3 \mapsto \R$ be a grid function defined by a sequence $(a_n)_n$. Then
\begin{enumerate}
\item there exist $(a_n)_n$ so that $D^-p_\phi(1)<0$, but $p_\phi(t)=\log 2$ for all $t\ge 1$;
\item there exist $(a_n)_n$  and $t_1>1$ so that $D^-p_\phi(1)<0$, $p_\phi(t)=\log 2$ for all $t\in [1, t_1]$ and $Dp_\phi(t)>0$ for all $t>t_1$;
\item there exist $(a_n)_n$ so that $Dp_\phi(t)<0$ for $t<1$, but $p_\phi$ is $C^1$ at $t=1$ and $p_\phi(t)=\log 2$ for all $t\ge 1$;
\item there exist $(a_n)_n$  and $t_1>1$ so that $Dp_\phi(t)<0$ for $t<1$, but $p_\phi$ is $C^1$ at $t=1$, and $p_\phi(t)=\log 2$ for all $t\in [1, t_1]$ and $Dp_\phi(t)>0$ for all $t>t_1$;
\end{enumerate}
\label{thm:main grid}
\end{teo}

\begin{figure}[h]
\begin{center}
\includegraphics[width=8cm]{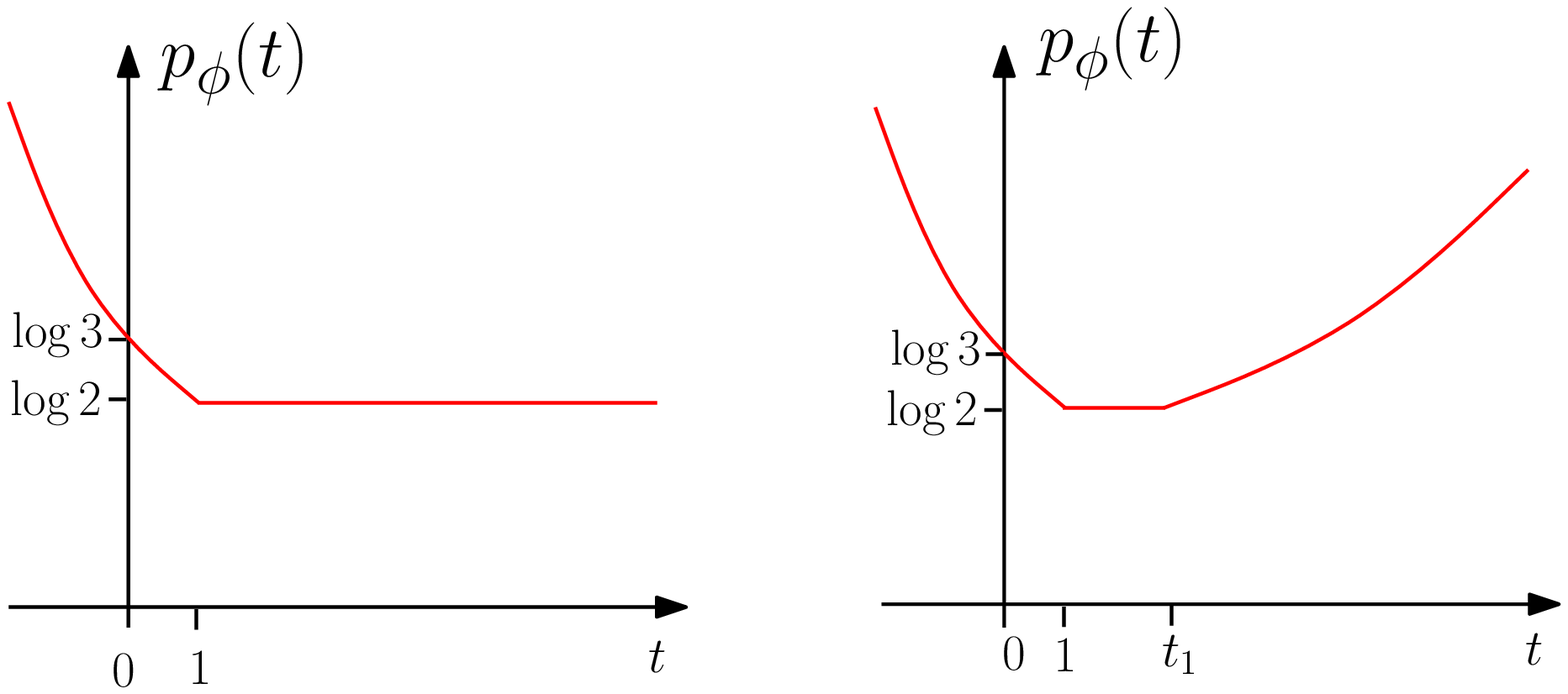}
\end{center}
\caption{\label{fig:DFU_kinks} Sketch of cases (1) and (2) of Theorem~\ref{thm:main grid}. }
\end{figure}

\begin{figure}[h]
\begin{center}
\includegraphics[width=8cm]{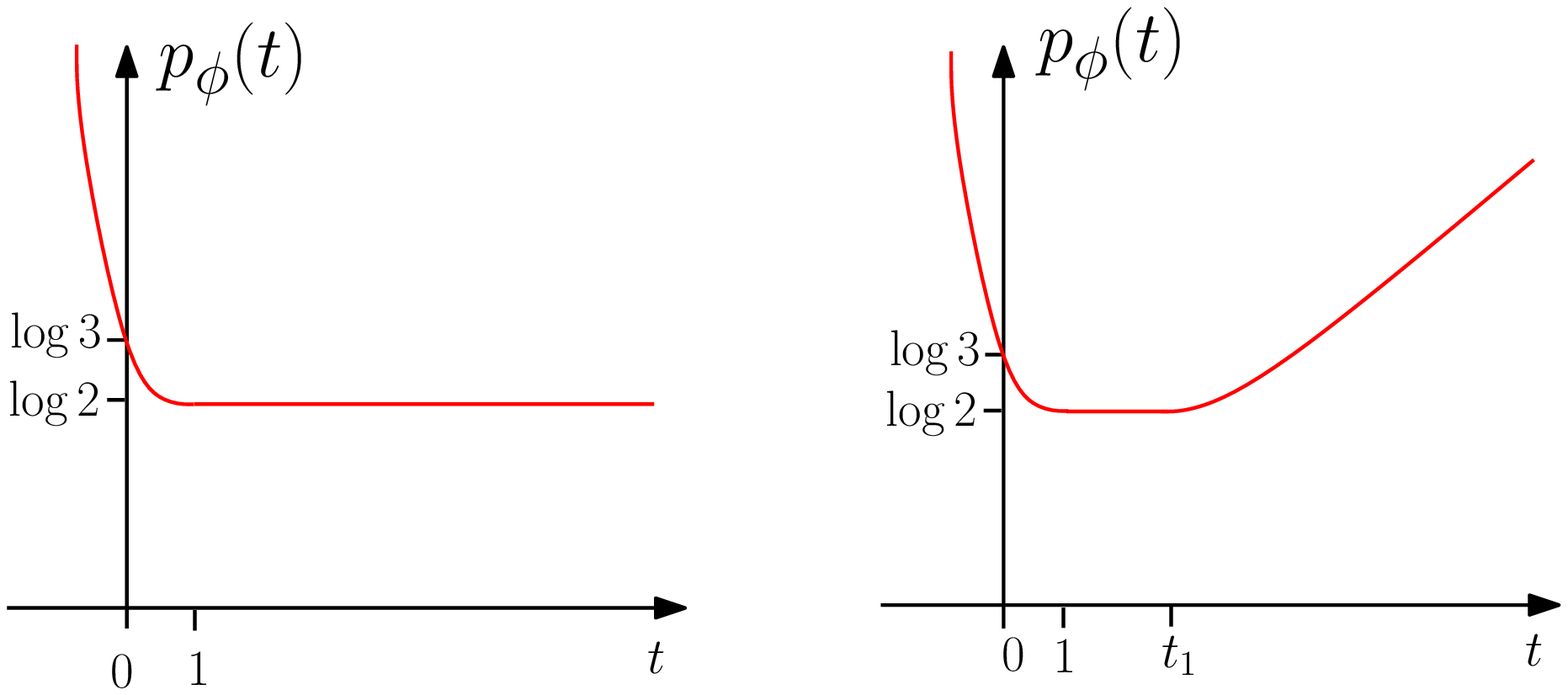}
\end{center}
\caption{\label{fig:DFU_flat} Sketch of cases (3) and (4) of Theorem~\ref{thm:main grid}.}
\end{figure}

We comment further on the systems $(\Sigma_3, \sigma, t\phi)$ with reference to Table~\ref{table:hof} (note that the only aspects which don't follow more or less immediately from the construction of our sequences $(a_n)_n$ are the null recurrent parts, which follow from Lemmas~\ref{lem:Dp 0} and \ref{lem:c1 or not}):

\begin{list}{$\bullet$}
{ \itemsep 1.0mm \topsep 0.0mm \leftmargin=7mm}
\item In case (1) of the theorem, the system is positive recurrent  for $t\le 1$ and transient for $t>1$.  The pressure function $p_\phi$ is DF.  See the left hand side of Figure~\ref{fig:DFU_kinks}.
\item In case (2) of the theorem, the system is positive recurrent  for $t\in (-\infty, 1]\cup [t_1, \infty)$ and transient for $t\in (1, t_1)$.  The pressure function is DFU.  See the right hand side of Figure~\ref{fig:DFU_kinks}.
\item In case (3) of the theorem, the system is positive recurrent for $t< 1$, null recurrent for $t=1$ and transient for $t>1$.    The pressure function is DF.   See the left hand side of Figure~\ref{fig:DFU_flat}.
\item In case (4) of the theorem, the system is positive recurrent  for $t\in (-\infty, 1)\cup (t_1, \infty)$, null recurrent for $t=1, t_1$ and transient for $t\in (1, t_1)$.  The pressure is DFU.  See the right hand side of Figure~\ref{fig:DFU_flat}.
\item  In the cases above where $t\phi$ is transient, as in Lemma~\ref{lem:Hof trans}, there is a dissipative $t\phi$-conformal measure.
\end{list}

The first and third parts of this theorem follow easily from the second and fourth parts, so we omit their proof.  The proof of this theorem occupies the rest of this section.

Note that in the case of Hofbauer's example, described in \ref{ssec:ho-ke}, the modes of recurrence of the potential and behaviour of the pressure function are  determined by the sums of the sequence $(a_n)_n$ (see Table \ref{table:hof}). This is also the case in our example.

\subsection{The inducing scheme}
The \emph{first return map}, denoted by $F$, is defined by
\begin{equation*}
F(x)= \sigma^{n}(x) \textrm { if } x \in X_n.
\end{equation*}
Note that the bad set $B$, on which $F$ is not defined, can be thought of as a coding for the middle third Cantor set. Also note that  the map
 $F$ restricted to the set  $X_n$  is a $2^{n-1}$-to-one map for each $n\ge1$.

The induced potential for this first return map is given by
$\Phi(x)= S_{\tau(x)}\phi(x)$.  For $n\ge 1$ we let
$$s_n:=a_0+\cdots+a_{n-1}.$$
Then by definition of $\phi$, for any $x\in X_n$ we have $\Phi(x)=s_n$.

The definitions of liftability of measures and aspects of inducing schemes for the case considered here are directly analogous to the setting considered in Section~\ref{sec:no conf}, so we do not give them here.

\begin{rem}
Note that the potential $\Phi$ is a locally constant over the countable Markov partition
$\bigcup_{n=1}^{\infty} X_n $. Therefore with the inducing procedure we have gained regularity on our potential. Observe that if there is a $\Phi$-conformal measure $\mu$, it is also a Gibbs measure, and indeed for any $x\in X_n$, $\mu(X_n)=e^{\Phi(x)}=e^{s_n}$.   Since $X_n$ consists of $2^{n-1}$ connected components, each has mass $2^{-(n-1)}e^{s_n}$.
\label{rmk:Phi loc cons}
 \end{rem}

Given a grid function (our potential) $\phi$ defined on $\Sigma_3$, to discuss equilibrium states for the induced system,  as in Section~\ref{sec:no conf}, it is convenient to shift the original potential to ensure that its induced version will have pressure zero.  Therefore, since we will be interested in the family of potentials $t\phi$ with $t \in \R$, we set
$$\psi_t:=t\phi-p_\phi(t) \text{ and } \Psi_t:=t\Phi-\tau p_\phi(t).$$
We denote $\Psi_t|_{X_i}$ by $\Psi_{t, i}$.

In this setting, the properties of the pressure function will depend directly on the choice of the sequence $(a_n)_n$.

Our first restriction on the sequence $(a_n)_n$ comes from a normalisation requirement. Indeed, for $t=1$ we want  $p_\phi(t)= P(\phi)=\log 2$. If this is the case, then the measure of maximal entropy $\mu_B$ on $B$ is an equilibrium measure since $h(\mu_B)+\int\phi~d\mu_B=h(\mu_B)=\log2$.  We choose $(a_n)_n$ so that
\begin{equation}1=\sum_ie^{\Psi_{1,i}}= \sum_{n\ge 1} 2^{n-1} e^{s_n-n\log 2}=\frac12\sum_{n\ge 1} e^{s_n}.\label{eq:t0 is 1}
\end{equation}

As in \cite[p25]{ba} for example, since $\Psi_t$ is locally constant, $e^{P(\Psi_1)}=\sum_ie^{\Psi_{1,i}}$, so \eqref{eq:t0 is 1} implies $P(\Psi_1)=0$.  Similarly,
\begin{equation}
e^{P(\Psi_t)}=\sum_{i \ge 1}e^{\Psi_{t,i}}= \sum_{i \ge 1} e^{t\Phi_{i}-\tau p_\phi(t)}= \sum_{n\ge 1} 2^{n-1} e^{ts_n-np_\phi(t)}.
\label{eq:ind eg}
\end{equation}

We have the following result, similarly to Lemma~\ref{lem:s=P}.

\begin{lema}
The existence of a conservative $\psi_t$-conformal measure $m_t$ implies $P(\Psi_t)=0$.
\label{lem:conf 0 pres}
\end{lema}

\begin{proof}
To apply the argument of Lemma~\ref{lem:s=P}, we only need to show that $$m_t\left(\left\{x:f^n(x)\in X_1 \text{ for infinitely many } n\right\}\right)>0.$$  If not then for $$A_k:=\left\{x\in X_1:f^{k+n}(x)\in X_0\cup X_2 \text{ for all } n\ge 1\right\},$$
$m_t(\cup_kA_k)>0$.  In particular, we have $m_t(A_0)>0$.  Observe that $A_0$ is a wandering set since $f^{-p}(A_0)\cap f^{-q}(A_0)=\es$ for positive $q\neq p$.  Therefore $m_t$ is not conservative.
\end{proof}

\subsection{Down-flat-up pressure occurs}

We first show that the DF case occurs and then show DFU occurs too. First we set $a_0=0$ and suppose that  for every $n \in \N$  the numbers  $a_n$ are chosen so that $a_n<0$ and \eqref{eq:t0 is 1} holds.

Since $\phi\le 0$; the pressure function $p_\phi$ is decreasing in $t$; $p_\phi(1)=\log2$; and $p_\phi(t)\ge \log 2$; this means that $p_\phi(t)=\log 2$ for all $t\ge 1$, so the DF case occurs.  The  transience of $(\Sigma, \sigma, t\phi)$ for $t>1$ follows as below.  In this case,
 $s_1=0$, and \eqref{eq:t0 is 1} can be rewritten as
\begin{equation}
1=\frac12\sum_{n\ge 1}e^{s_n} = \frac12\left(1+\sum_{n\ge 2}e^{s_n}\right).
\label{eq:DF}
\end{equation}

Now to show that the DFU case occurs, let us first set $(a_n)_n$ as above.
Next we replace $a_0$ by $\tilde a_0:=\delta\in (0,\log2)$, and $a_1$ by $\tilde a_1:=a_1+\delta'$, where $\delta'<0$ is such that \eqref{eq:DF} still holds when $(s_n)_n$ is replaced by $(\tilde s_n)_n$, the rest of the $a_n$ being kept fixed.   So \eqref{eq:DF} implies that
$$\frac12 e^\delta +\frac12 e^{\delta+\delta'} =1,$$
so $P(\Psi_1)=0$.
Using Taylor series, we have $2\delta+\delta'<0$.
We now replace $\phi,\ \Phi$ and $\Psi_t$ by the adjusted potentials $\tilde\phi$, $\tilde\Phi$, $\tilde\Psi_t$.

\begin{lema}
There exists $t_1>1$ such that $P(\tilde\Psi_t)<0$ for all $t\in (1,t_1)$.
\label{lem:flat}
\end{lema}

Since $p_{\tilde\phi}(t)\ge \log 2$, the lemma implies that $p_{\tilde\phi}(t)=\log 2$ for $t\in [1,t_1]$, so the DF property of the pressure function persists under our perturbation of $\phi$ to $\tilde\phi$.  
Moreover, Lemma~\ref{lem:conf 0 pres} implies that $(\Sigma, \sigma, t\tilde\phi)$ is transient for $t\in [1,t_1]$.

The `up' part of the DFU property, must hold for $p_{\tilde\phi}$ since $\tilde a_0>0$: indeed the graph of $p_{\tilde\phi}$ must be asymptotic to $t\mapsto \tilde a_0t$, and the equilibrium measures for $t\tilde\phi$ denoted by $\mu_t$ must tend to the Dirac measure on the fixed point in $[2]$.

\begin{proof}[Proof of Lemma~\ref{lem:flat}]
As above, since $\tilde \Psi_t$ is locally constant, $e^{P(\tilde\Psi_t)}$ can be computed as
\begin{equation}
e^{P(\tilde\Psi_t)}=\sum_{i \ge 1} e^{t\tilde\Phi_{i}-\tau_i p_{\tilde\phi}(t)}= \sum_{n\ge 1} 2^{n-1} e^{t\tilde s_n-np_{\tilde\phi}(t)}.
\label{eq:trans press}
\end{equation}

Since $p_{\tilde\phi}(t)\ge \log 2$ and $s_n<0$ for $n\ge 2$ for $t>1$ close to 1 we have
\begin{align*}
\sum_{i \ge 1} e^{t\tilde\Phi_{i}-\tau_i p_{\tilde\phi}(t)} &\le \frac12\sum_{n\ge 1} e^{t\tilde s_n}=\frac12\left(e^{t\delta} +e^{t(\delta+\delta')}\sum_{n\ge 2} e^{ts_n}\right)<\frac12\left(e^{t\delta}+e^{t(\delta+\delta')}\right)<1,
\end{align*}
where the final inequality follows from a Taylor series expansion and the fact that $2\delta+\delta'<0$.  This implies that
$P(\tilde\Psi_t)<0$ for $t>1$ close to 1.  We let $t_1>t''>1$ be minimal such that $t>t_1$ implies $p_{\tilde\phi}(t)>\log 2$.
\end{proof}

For brevity, from here on we will drop the tildes from our notation when discussing the potentials above.

\subsection{Tails and smoothness}

So far we have not made any assumptions on the precise form of $a_n$ for large $n$.  In this section we will make our assumptions precise in order to distinguish cases (1) from case (3) in Theorem~\ref{thm:main grid}, as well as case (2) from case (4).  That is to say, we will address the question of the smoothness of $p_\phi$ at 1 and $t_1$ by defining different forms that $a_n$, and hence $s_n$, can take as $n\to \infty$.   In fact, it is only the form of $a_n$ for large $n$ which separates the cases we consider. As in \cite[Section 4]{Hnonuni}, see also \cite[Section 6]{BTeqgen}, let us assume that for all large $n$, for some $\gamma>1$ we have
$$a_n=\gamma\log\left(\frac{n}{n+1}\right).$$

We will see that we have a first order phase transition  in the pressure function $p_\phi$ whenever $\gamma>2$, but not when $\gamma\in (1,2]$.

Clearly, there is some $\kappa\in \R$ so that
$s_n\sim \kappa-\gamma\log n.$
So applying the computation in \eqref{eq:t0 is 1},
$$\sum_ie^{\Psi_{1,i}}=\frac12\sum_ne^{s_n}=(1+O(1))\sum_n\frac1{n^\gamma}.$$
Since we assumed that $\gamma>1$, we can ensure that this is finite, and indeed we can choose $(a_n)_n$ in such a way that $\sum_i e^{\Psi_{1,i}}=1$ as in \eqref{eq:t0 is 1}.

We now show that the graph of the pressure in the case that the pressure is DFU is either $C^1$ everywhere or only non-$C^1$ at both $t=1$ and $t=t_1$.   The issue of smoothness of $p_\phi$ in the $DF$ case follows as in the DFU case, so Theorem~\ref{thm:main grid} then follows from Lemma~\ref{lem:flat} and the following proposition.

\begin{prop}
For potential $\phi$ chosen as above, there exists $t_1>1$ such that $p_\phi(t)=\log 2$ for all $t\in [1,t_1]$.  Moreover, if $\gamma\in (1,2]$ then $p_\phi$ is everywhere $C^1$, while if $\gamma>2$ then $p_\phi$ fails to be differentiable at both $t=1$ and $t=t_1$.
\end{prop}

The first part of the proposition follows from Lemma~\ref{lem:flat}, while the second follows directly from the following two lemmas.  We will use the fact that if $p_\phi$ is $C^1$ at $t$ then $Dp_\phi(t)=\int\phi~d\mu_t$ (see \cite[Chapter 4]{PU}).

\begin{lema}
If $\gamma\in (1, 2]$ then $Dp_\phi(1)=0$.
\label{lem:Dp 0}
\end{lema}

\begin{proof}
Since by Lemma~\ref{lem:flat}, for $t\in [1,t_1]$, $p_\phi$ is constant $\log 2$, we have $Dp_\phi^+(1)=0$, so to prove $Dp_\phi(1)=0$ we must show $Dp_\phi^-(1)=0$.

Suppose that $t<1$.  Then by the Abramov formula (see Remark~\ref{rmk:Abram}),
\begin{align}\int\phi~d\mu_t&=\frac{\int\Phi~d\mu_{\Psi_t}}{\int\tau~d\mu_{\Psi_t}}=
\frac{\sum_ns_n e^{ts_n-n(p_\phi(t)-\log2)}}{2\sum_nn e^{ts_n-n(p_\phi(t)-\log2)}}.
\label{eq:Abra flat}
\end{align}

 As above, for large $n$, $s_n\sim \kappa-\gamma\log n$, which is eventually much smaller, in absolute value, than $n$.  Since also, $\sum_nn e^{s_n-n(p_\phi(t)-\log2)}\to \infty$ as $t\to 1$, we can make $\int\phi~d\mu_t$ arbitrarily small by taking $t<1$ close enough to 1.  Since when $p_\phi$ is $C^1$ at $t$ then $Dp_\phi(t)=\int\phi~d\mu_t$, this completes the proof.
\end{proof}

\begin{lema}
Suppose that $\phi$ is a grid function as above and the pressure $p_\phi$ is DFU.  Then $p_\phi$ is $C^1$ at $t=1$ if and only if $p_\phi$ is $C^1$ at $t=t_1$.
\label{lem:c1 or not}
\end{lema}

\begin{proof}
Using the argument in the proof of Lemma~\ref{lem:Dp 0}, in particular \eqref{eq:Abra flat}, if $\int\tau~d\mu_{\Psi_{t}}=\infty$,
we can make $Dp_\phi(t')$ arbitrarily close to 0 by taking $t'$ close enough to $t$.  Similarly if this integral is finite at $t$ then the derivative $Dp_\phi(t)$ is non-zero.  So to prove the lemma, we need to show that the finiteness or otherwise of $\int\tau~d\mu_{\Psi_{t}}$ is the same at both $t= 1$ and $t=t_1$.

As in the proof of Lemma~\ref{lem:flat}, $s_n<0$ for $n\ge 2$. So
since $p_\phi(t)=p_\phi(t_1)$ and $t_1>1$,
$$\int\tau~d\mu_{\Psi_1}>\sum_{i\ge 2}\tau_ie^{\Phi_i-\tau_i p_\phi(1)} >
\sum_{i\ge 2}\tau_ie^{t_1\Phi_i-\tau_i p_\phi(t_1)}.$$
Therefore if $\int\tau~d\mu_{\Psi_1}<\infty$ then $\int\tau~d\mu_{\Psi_{t_1}}<\infty$.  Similarly, if $\int\tau~d\mu_{\Psi_{t_1}}=\infty$ then $\int\tau~d\mu_{\Psi_{1}}=\infty$.
Hence either $Dp_\phi(1)$ and $Dp_\phi(t_1)$ are both 0 or are both non-zero.
\end{proof}

\begin{rem}
In the case $\gamma\in (1,2]$, the measure $\mu_{\Psi_1}$
is not regarded as an equilibrium state for the system $(M_0, F, \Psi_1)$ since $$\int\Psi_1~d\mu_{\Psi_1}=-\infty.$$

This follows since
$$\int\Psi_1~d\mu_{\Psi_1}=\sum_n(s_n-np_\phi(t))e^{s_n} \asymp \sum_n\frac{a-\gamma\log n-np_\phi(t)}{n^\gamma},$$ so for all large $n$ the summands are dominated by the terms $-p_\phi(t)n^{1-\gamma}$ which are not summable.
\end{rem}

\begin{rem}
If we wanted the limit of $\mu_t$ as $t\to \infty$ to be a measure with positive entropy, then one way would be to choose our dynamics to be $x\mapsto 5x \mod 1$ and the set $M_0$ to correspond to the interval $[0, 2/5]$ for example.
\end{rem}

Note that for our examples, we can not produce more than two equilibrium states simultaneously.   One can see this as following since we are essentially working with two intermingled systems.

\section*{Acknowledgements}
We would like to thank H.\ Bruin and O.\ Sarig for their useful comments and remarks.  This work was begun while MT was at the
Centro de Matem\'atica da Universidade do Porto, which he would like to thank for their support.

\end{document}